\documentclass[11pt]{amsart}
\usepackage{amsmath,amssymb}
\usepackage{amsfonts}
\usepackage{eucal}
\usepackage{amsthm}
\numberwithin{equation}{section}
\usepackage{comment}
\newcommand{\jap}[1]{\langle #1 \rangle}

\def\a{\alpha}
\def\b{\beta}
\def\c{\gamma}
\def\d{\delta}
\def\e{\varepsilon}
\def\f{\varphi}
\def\g{\psi}

\def\i{\mbox{\raisebox{.5ex}{$\chi$}}}

\def\l{\lambda}
\def\m{\mu}
\def\n{\nu}

\def\x{\xi}
\def\y{\eta}
\def\z{\zeta}

\newcommand{\G}{\Psi}
\renewcommand{\L}{\Lambda}
\renewcommand{\O}{\Omega}

\newcommand{\Op}{\mathrm{Op}}

\def\re{\mathbb{R}}
\def\co{\mathbb{C}}
\def\ze{\mathbb{Z}}

\def\pa{\partial}

\def\coloneqq{\mathrel{\mathop:}=}%

\renewcommand{\Im}{\text{{\rm Im}\;}}

\newcommand{\supp}{\text{{\rm supp}\;}}

\newcommand{\Ker}{\text{{\rm Ker}\;}}

\newcommand{\oi}{\overline\i}
\newcommand{\norm}[1]{\| #1 \|}
\newcommand{\bignorm}[1]{\bigl\| #1 \bigr\|}

\newcommand{\bigpare}[1]{\bigl(#1\bigr)}
\newcommand{\biggpare}[1]{\biggl(#1\biggr)}

\newcommand{\bigbra}[1]{\bigl\{#1\bigr\}}

\newcommand{\bigset}[2]{\bigl\{#1\bigm|#2\bigr\}}

\newcommand{\bigjap}[1]{\bigl\langle #1 \bigr\rangle}

\newcommand{\biggjap}[1]{\biggl\langle #1 \biggr\rangle}

\newcommand{\bigabs}[1]{\bigl| #1 \bigr|}

\newcommand{\biggabs}[1]{\biggl| #1 \biggr|}

\newtheorem{thm}{Theorem}[section]
\newtheorem{lem}[thm]{Lemma}
\newtheorem{prop}[thm]{Proposition}
\newtheorem{cor}[thm]{Corollary}

\theoremstyle{definition}

\newtheorem{ass}{Assumption}

\theoremstyle{remark}
\newtheorem{rem}[thm]{Remark}

\title[Essential self-adjointness of real principal type operators]{Essential self-adjointness of \\ real principal type operators}
\author{Shu Nakamura}
\address{Department of Mathematics, Faculty of Sciences, Gakushuin University, 1-5-1, Mejiro, Toshima, Tokyo, Japan 171-8588}
\email{shu.nakamura@gakushuin.ac.jp}

\author{Kouichi Taira}
\address{Graduate School of Mathematical Sciences, the University of Tokyo, 3-8-1 Komaba, Meguro, Tokyo, Japan 153-8914}
\email{taira@ms.u-tokyo.ac.jp}

\begin{document}
\maketitle

\begin{abstract}
We study the essential self-adjointness for real principal type differential operators. Unlike the elliptic case, 
we need geometric conditions even for operators on the Euclidean space with asymptotically constant 
coefficients, and we prove the essential self-adjointness under the null non-trapping condition.
\end{abstract}


\section{Introduction}
In this paper, we consider formally self-adjoint real principal type operator $P=\Op(p)$ on the Euclidean space $\re^n$ with $n\geq 1$, where $\Op(\cdot)$ denotes the Weyl quantization. 
A typical example is the Klein-Gordon operator with variable coefficients (see Remark~\ref{basicexample}), and 
the propagation of singularities plays an essential role in the proof of the essential self-adjointness. 

We suppose the symbol $p(x,\x)$ is real principal type with asymptotically constant coefficients in the following sense: 

\begin{ass}\label{assa}
Let $m\geq 2$, $p, p_m\in C^{\infty}(\re^{2n})$ and $p_0\in C^{\infty}(\re^n)$ be real-valued functions of the form
\[
p(x,\x)=\sum_{|\a|\leq m}a_{\a}(x)\x^{\a},\,\, p_m(x,\x)=\sum_{|\a|=m}a_{\a}(x)\x^{\a},\,\, p_0(\x)=\sum_{|\a|=m}b_{\a}\x^{\a}
\]
where $b_{\a}\in \re$ and $a_{\a}\in C^{\infty}(\re^n)$ such that for any multi-index $\a\in\ze_+^n$, 
\[
|\pa_{x}^{\b}(a_{\a}(x)-b_{\a})|\leq C_{\b}\jap{x}^{-\m-|\b|}, \quad x\in\re^n
\]
with some $\m>0$, where we set $b_{\a}=0$ for $|\a|\leq m-1$. Moreover, there exists $C>0$ such that
\[
C^{-1}|\x|^{m-1}\leq |\pa_{\x}p_0(\x)|\leq C|\x|^{m-1},\,\, C^{-1}|\x|^{m-1}\leq |\pa_{\x}p_m(x,\x)|\leq C|\x|^{m-1}
\]
for $(x,\x)\in \re^{2n}$.
\end{ass}

Let $(y(t), \y(t))=(y(t, x_0, \x_0),\y(t, x_0, \x_0))\in C^1(\re\times\re^{2n};\re^{2n})$ be the solution
to the Hamilton equation:
\[
\frac{d}{dt} y(t) =\frac{\pa p_m}{\pa \x}(y(t),\y(t)), \quad
\frac{d}{dt} \y(t) =-\frac{\pa p_m}{\pa x}(y(t),\y(t)), \quad
t\in\re, 
\]
with the initial condition: $(y(0),\y(0))=(x_0,\x_0)\in\re^{2n}$. We suppose the following null non-trapping condition: 

\begin{ass}\label{assnull}
For any $(x_0, \x_0)\in p_m^{-1}(0)$ with $\x_0\neq 0$, 
$| y(t, x_0, \x_0) | \to \infty$ as $|t|\to \infty$. 
\end{ass}

Our main theorem is the following: 

\begin{thm}\label{nullthm}
Suppose  Assumption \ref{assa} and \ref{assnull}. Then $P=\Op(p)$ is essentially self-adjoint on $C_c^\infty(\re^n)$. 
\end{thm}

\begin{rem}\label{basicexample}
(Klein-Gordon operators on asymptotically Minkowski spaces) 
Let $g_0$ be the Minkowski metric on $\re^n$: $g_0=dx_1^2-dx_2^2-...-dx_n^2$ and 
$g^{-1}=(g^{ij}_0)_{i,j=1}^n$ be its dual metric. 
A Lorentzian metric $g$ on $\re^n$ is called asymptotically Minkowski if $g^{-1}(x)=(g^{ij}(x))_{i,j=1}^n$ satisfies, 
for any $\a\in\ze_+^n$ there is $C_\a>0$ such that 
\[
|\pa_{x}^{\a}(g^{ij}(x)- g_0^{ij})|\leq C_{\a}\jap{x}^{-\m-|\a|},\quad x\in\re^n, 
\]
with some $\m>0$. Suppose  $V(x), A_j(x)\in C^\infty(\re^n;\re)$, $j=1,\dots, n$,  such that 
\[
|\pa_{x}^{\a}V(x)|\leq C_{\a}\jap{x}^{-\m-|\a|}, \quad
|\pa_{x}^{\a}A_j(x)|\leq C_{\a}\jap{x}^{-\m-|\a|}, 
\quad x\in\re^n, 
\]
for any $\a\in\ze_+^n$ . 
Then the symbol 
\[
p(x,\x)=\sum_{j,k=1}^ng^{jk}(x)(\x_j-A_j(x))(\x_k-A_k(x)) +V(x)
\]
satisfies Assumption \ref{assa}. 
The essential self-adjointness for this model is studied by Vasy~\cite{Va1}. 
\end{rem}

\begin{rem}
In this paper, we only deal with operators with order greater than $1$.  
The essential self-adjointness of first order operators on $C_c^{\infty}(\re^n)$ can be proved by
 Nelson's commutator theorem with its conjugate operator $N=-\Delta+|x|^2+1$ (\cite[Theorem X.36]{RS}). 
We also note that if  $P$ commutes with the complex conjugation: $P\overline{u}=\overline{Pu}$, then, it is enough to assume 
the forward null non-trapping condition only instead of null non-trapping condition (cf.\ \cite[Theorem X.3]{RS}). 
\end{rem}


The study of essential self-adjointness has a long history but mostly on operators of elliptic type 
(see \cite{RS} Chapter~X and reference therein). 
For the construction of solutions to evolution equation with real principal type operators, we refer the classical paper \cite{DH} by Duistermaat and H\"ormander, and the textbook by H\"ormander \cite{Ho}. 
Chihara \cite{C} studies the well-posedness and the local smoothing effects of the Schr\"odinger-type equations :
$\pa_t u(t,x)=-iPu(t,x)$ under the globally non-trapping condition. The well-posedness  implies essential self-adjointness of $P$ if the operator $P$ is symmetric. We assume the non-trapping condition only for null trajectories, since the microlocally 
elliptic region should not be relevant. 

Recently, the scattering theory for Klein-Gordon operators on Lorenzian manifolds has been studied by several authors 
(see, e.g., \cite{BV,GNV,Va1} and references therein). We also mention related work on Strichartz estimates for Lorenzian manifolds (\cite{GT, MT, Taira}),  nonlinear Schr\"odinger-type equations with Minkowski metric (\cite{GS, S, W}), 
and quantum field theory on Minkowski spaces (\cite{VW, GW}). 
In order to study spectral properties of such equations or operators, self-adjointness is fundamental. 
We note a sufficient condition for the essential self-adjointness is discussed in Taira \cite{Taira}. 
The essential self-adjointness for Klein-Gordon operators on scattering Lorentzian manifolds is proved by Vasy  \cite{Va1} under the same null non-trapping condition. We had independently found a proof of the essential self-adjointness 
using different method for compactly supported perturbations (we discuss the basic idea in Appendix \ref{appcpt}). 
Inspired by discussions with Vasy during 2017, we generalized the model to include long-range perturbations, 
and also to higher order real principal type operators. Our proof is considerably different from \cite{Va1}, 
relatively self-contained, and hopefully simpler even though our result is more general than \cite{Va1} 
for the $\re^n$ case. 

This paper is constructed as follows: In Section~\ref{pre}, we prepare several notations and basic lemmas. 
Our main result is proved in Section~\ref{section-proof}. In Subsection~\ref{local} we show that $(P-i)u=0$ implies 
$u$ is smooth. The basic idea of the proof is analogous to Nakamura \cite{N} on microlocal smoothing estimates, 
and relies on the construction of time-global escaping functions (see also Ito, Nakamura \cite{IN} for 
related results for scattering manifolds). 
The technical detail is given in Appendix~\ref{appb}. 
In Subsection~\ref{global}, we show the local smoothness implies an weighted Sobolev estimate, which is sufficient 
for the proof of the essential self-adjointness. The idea is analogous to the {\em radial point estimates}\/ of 
Melrose \cite{M}, and also related to the positive commutators method of Mourre. Here we construct weight functions 
explicitly to show necessary operator inequalities. The proof relies on the standard pseudodifferential operator calculus. 
In Appendix~\ref{appa}, we prove non-trapping estimates for the classical trajectories generated by $p_m(x,\x)$, 
which are necessary in Appendix~\ref{appb}.  
The main lemma (Lemma~\ref{energy}) is a generalization of a result by Kenig, Ponce, Rolvung and Vega \cite{KPRV}, 
though the proof is significantly simplified. In Appendix~\ref{appcpt}, we give a simplified proof of the essential self-adjintness 
for the compactly supported perturbation case. In this case the relatively involved argument of Subsection~\ref{global} 
is not necessarily. 

\medskip
\noindent
\textbf{Acknowledgment.}  
SN is partially supported by JSPS grant Kiban-B 15H03622. 
KT is supported by JSPS Research Fellowship for Young Scientists, KAKENHI Grant Number 17J04478 and the program FMSP at the Graduate School of Mathematics Sciences, the University of Tokyo.
We are grateful to Andra\'s Vasy for stimulating discussions during RIMS meeting at Kyoto in 2017.


\section{Preliminary}\label{pre}
We set $\jap{x}=(1+|x|^2)^{1/2}$ and $D_x=-i\pa/\pa x$. We denote the weighted Sobolev spaces by 
\[
H^{s,\ell} = H^{s,\ell}(\re^n) = \jap{x}^{-\ell}\jap{D_x}^{-s}[L^2(\re^n)], 
\]
for $s,\ell\in\re$, and their norms are given by 
\[
\norm{\f}_{H^{s,\ell}} = \bignorm{\jap{D_x}^s \jap{x}^\ell \f}_{L^2}. 
\]

We use the following notation of pseudo-differential operators. 
For any symbol $a\in C^{\infty}(\re^{2n})$, we define the Weyl quantization of $a$ 
(at least formally) by
\[
\Op(a) u(x) =(2\pi)^{-n} \iint e^{i(x-y)\cdot\x} a((x+y)/2,\x) u(y) \,dy\,d\x, \quad u\in\mathcal{S}(\re^n).
\]
We set the symbol classes $S$, $S^{k,\ell}$ and $S(m,g)$ by
\begin{align*}
&S \coloneqq \bigset{a\in C^{\infty}(\re^{2n})}{\forall\,\a, \b,\;\exists\, C_{\a\b} \, :\, |\pa_x^{\a}\pa_{\x}^{\b}a(x,\x)| \leq C_{\a\b}},\\
&S^{k,\ell} \coloneqq \bigset{a\in C^{\infty}(\re^{2n})}{\forall\,\a, \b,\;\exists\, C_{\a\b} \,:\, |\pa_x^{\a}\pa_{\x}^{\b}a(x,\x)| \leq C_{\a\b}\jap{x}^{k-|\a|}\jap{\x}^{\ell-|\b|}},\\
& S(m,g) \coloneqq \bigset{a\in C^{\infty}(\re^{2n})}{\text{for vector fields}: X_1,X_2,\dots ,X_k, 
\exists\, C \text{ such that}\\
& \qquad\quad 
 |X_1X_2...X_ka(x,\x)| \leq Cm(x,\x)g(X_1,X_1)^{\frac{1}{2}}g(X_2,X_2)^{\frac{1}{2}}\cdots g(X_k,X_k)^{\frac{1}{2}}},
 \end{align*}
where $g$ is a slowing varying metric and $m$ is a $g$-continuous function (see \cite[\S 18.4]{Ho}). 
We denote the Poisson bracket of symbols $a$ and $b$ by $\{a,b\} =\pa_x a\cdot\pa_\x b -\pa_\x a\cdot\pa_x b$.

The proofs of the following lemmas are standard, and we omit the proofs. 

\begin{lem}\label{propwf}
Let $(x_0, \x_0)\in \re^{2n}$ with $\x_0 \neq 0$.  Suppose that there exists $a\in S$ such that $a(x_0, \x_0)>0$ and 
$\| \Op(a_h)u \|_{L^2} = O(h^{k+\e})$ for some $\e>0$, where $a_h(x,\x)=a(x,h\x)$. Then $u\in H^k$ microlocally at $(x_0,\x_0)$.
\end{lem}

\begin{lem}\label{convergence}
Let $k,\ell\in \re$. Assume $a_j\in S^{k,\ell}$ is a bounded sequence in $ S^{k,\ell}$ and $a_j\to 0$ in $S^{k+\d,\ell+\d}$ for some $\d>0$. Then, for each $s, t\in \re$ and $u\in H^{s,t}$
\[
\|\Op(a_j)u\|_{H^{s-k,t-\ell}}\to 0\quad\text{as } j\to \infty.
\]
\end{lem}


\section{Proof of Theorem \ref{nullthm}}\label{section-proof}

By the basic criterion for the essential self-adjointness (\cite[Theorem~VII.3]{RS}), it is sufficient to show
\[
\Ker(P^* \pm i) = \bigr\{0\bigr\}
\]
to prove Theorem \ref{nullthm}. 
Since $D(P) = C_c^ \infty(\re^n)$, we have $D(P^{*})=\{u\in L^2(\re^n)\,|\, Pu\in L^2(\re^n)\}$
where $P$ acts on $u$ in the distribution sense. We hence show: 
\[
(P \pm i)u = 0 \text{ in }\mathcal{D}'(\re^n) \text{ for } u\in L^2(\re^n)
\text{ implies } u=0. 
\]
We only consider ``$-$'' case. The ``$+$'' case is similarly handled.
Moreover, we note if $u$ satisfies $(P-i)u=0$ and $u\in H^{\frac{m-1}{2},-\frac{1}{2}}(\re^n)$, 
then $u=0$ follows from a simple argument in \cite{Va1}.
Namely, we take a real-valued function $\g\in C_c^{\infty}(\{t\in \re\,|\, t\leq 2\})$ such that $\g(t)=1$ for $t\leq1$ and set $\g_{R}(x,\x)=\g(\jap{x}/R)\g(\jap{\x}/R)$. Then we have
\[
-2i\|u\|_{L^2}^2=(Pu,u)_{L^2}-(u,Pu)_{L^2}=\lim_{R\to \infty}([\Op(\g_R),P]u,u)_{L^2}.
\] 
We note that $[\Op(\g_R),P]$ is uniformly bounded in $\Op S^{m-1,-1}$ and converges to $0$ in $\Op S^{m-1+\d,-1+\d}$ as $R\to \infty$ for any $\d>0$. We obtain $u=0$ by using Lemma \ref{convergence}.
Thus, in order to prove Theorem \ref{nullthm}, it suffices to prove
\begin{prop}\label{reg}
If $u\in L^2(\re^n)$ satisfies $(P-i)u=0$, then $u \in H^{\frac{m-1}{2},-\frac{1}{2}}$.
\end{prop}

The proof of Proposition \ref{reg} is divided into two parts. In Subsection \ref{local}, we prove the local smoothness of $u$. In Subsection \ref{global}, using the local smoothness of $u$, we prove weighted Sobolev properties of $u$.

\subsection{Local regularity}\label{local}

The main result of this subsection is the following proposition. 
We note that we need the null non-trapping condition only for this proposition. 
\begin{prop}\label{locre}
If $u\in L^2(\re^n)$ satisfies $(P-i)u=0$, then $u\in C^{\infty}(\re^n)$.
\end{prop}

\begin{proof}
It suffices to prove $u\in H^k_\mathrm{loc}(\re^n)$ for any $k>0$. We use the contradiction argument. Suppose $u\notin H_\mathrm{loc}^{k}(\re^n)$ with some $k$.  By Lemma \ref{propwf}, there exist $(x_0, \x_0)\in \re^n \times \re^n$ with $\x_0 \ne 0$, $C>0$, and a sequence $\{h_\ell\}\subset (0, 1]$ such that for any $a\in C_0^\infty(\re^n)$ with $a(x_0, \x_0) = 1 $, 
\[
h_{\ell} \to 0  \text{ as }  \ell \to \infty,  \text{ and }\| \Op(a_{h_{\ell},m})u \| \geq Ch_{\ell}^{\frac{k}{m-1}+1},
\]
where $a_{h,m}(x,\x)=a(x,h^{\frac{1}{m-1}}\x)$. 
We may assume $(x_0,\x_0) \in p_m^{-1}(0)$ since $u$ is smooth microlocally in $\re^{2n}\setminus p_m^{-1}(\{0\})$. Now we use the following proposition.

\begin{prop}\label{singprop}

There exists a family of bounded operators $\{F(h,t)\}_{0<h\leq 1, t\geq 0}$ on $L^2(\re^n)$ such that
\begin{enumerate}
\item $F(h,0)=\Op(\g_{h})^2=\Op(\g_h)^*\Op(\g_h)$, where $\g_h$ satisfies $\g_h(x_0,\x_0)\geq 1$ 
and for any $\a,\b\in\ze_+^n$, 
\[
|\pa_{x}^{\a}\pa_{\x}^{\b}\g_{h}(x,\x)|\leq C_{\a\b}h^{\frac{|\b|}{m-1}}\jap{x}^{-|\a|}.
\]
\item There exists $C>0$ such that for $0<h\leq 1$, 
\[
\|F(h,t)\|_{B(L^2)}\leq C\jap{t}h^{(-m+2)/(m-1)}, \quad t\geq 0. 
\]
\item There exists $R(h,t)\in B(L^2(\re^n))$ such that
\begin{align*}
&\frac{d}{dt}F(h,t)+i[P,F(h,t)]\geq -R(h,t),\quad t\geq 0, \\
& \sup_{t\geq 0}\jap{t}^{-1}\|R(h,t)\|_{B(L^2)}=O(h^{\infty})
\quad\text{as }h\to 0. 
\end{align*}
\end{enumerate}
\end{prop}

Proposition \ref{singprop} can be proved similarly as \cite[Lemma 9]{N}. 
For the completeness, we give a proof of Proposition \ref{singprop} in the Appendix \ref{appb}. 
Now we set $u(t, x) \coloneqq e^{-t}u(x) $. Then $u(t,x)$ satisfies
\[
i\pa_tu(t,x) - Pu(t,x) = 0,\,\, \|u(t) \|_{L^2(\re^n)} \leq e^{-t}\|u\|_{L^2(\re^n)},
\]
where the first equality is in the distributional sense. 
We set $F_{\ell}(t) = F(h_{\ell},t)$. Then, we have
\begin{align*}
Ch_{\ell}^{\frac{2k}{m-1}+2} \leq & \|\Op(\g_{h_{\ell}})u\|^2=(u, F_{\ell}(0)u)\\
= & (u(t), F_{\ell}(t)u(t))- \int_0^t \frac{d}{ds}\bigpare{u(s), F_{\ell}(s)u(s)}ds\\
= &  (u(t), F_{\ell}(t)u(t)) - \int_0^t \biggpare{u(s), \biggpare{\frac{dF_{\ell}}{ds}(s) + i[P, F_{\ell}(s)]}u(s)} ds\\
\leq & Ch_{\ell}^{\frac{-m+2}{m-1}}\jap{t}e^{-2t}\|u\|^2 
+ O(h_{\ell}^{\infty})\cdot \|u\|^2 \int_0^te^{-2s}\jap{s}ds,
\end{align*}
where all the inner products and norms here are in $L^2(\re^n)$, and $O(h_{\ell}^{\infty})$ is uniformly in $t$. Now, we take $t = h_{\ell}^{-1}$ then we conclude a contradiction. 
Thus, we obtain $u\in H^k_\mathrm{loc}(\re^n)$ for any $k>0$. This completes the proof of Proposition \ref{locre} 
\end{proof}

\subsection{Uniform regularity outside a compact set}\label{global}

In this subsection, we prove a priori sub-elliptic estimates near infinity. 
The following estimates are based on the radial points estimates in \cite{M}, where the radial points estimates 
are used for scattering theory on scattering manifolds. 
By the classical propagation of singularities, the singularities of a solution to $Pu=0$ 
(provided $P$ is real-valued real principal type) propagate along the Hamilton flow associated with $p$. 
At points where the Hamilton vector filed vanishes, we may use the so-called radial points,
which implies $u$ is rapidly decaying at a radial source if $u$ has a threshold regularity at the radial source. 

In our case, the radial points estimates are analogous to the  Mourre estimate microlocally near outgoing or 
incoming regions, which is used commonly in scattering theory. We give a self-contained proof 
of the radial point estimate based on an explicit construction of escaping functions. 
We note the operator theoretical framework of the Mourre theory is not applicable here since 
we do not have the self-adjointness of $P$ at this point. 

We set 
\[
P=P_0+Q, \quad P_0=p_0(D_x), \quad Q=\Op(q),
\]
where
\[
q(x,\x)=p(x,\x)-p_0(\x)\in S^{m,-\m},\,\, V(x,\x)=p(x,\x)-p_m(x,\x)\in S^{m-1,-\m}. 
\]
We use the following smooth cut-off functions: Let $\i\in C^\infty(\re)$ be such that 
\[
\i(t)=\begin{cases} 1 \ &\text{if }t\leq 1, \\ 0 \ &\text{if }t\geq 2, \end{cases}
\quad 0\leq \i(t)\leq 1, \quad \i'(t)\leq 0\  \text{ for }t\in\re,
\]
and $\supp\chi'\Subset(1,2)$. 
We write $\oi(t)=1-\i(t)$, and 
\[
\i_M(x) =\i(|x|/M), \quad \oi_M(x)=\oi(|x|/M), 
\quad x\in\re^n, 
\]
with $M>0$. 
A main result of this subsection is the following theorem.

\begin{thm}\label{rad}
Let $\c>0$ and $z\in\co\setminus\re$. There is $M>0$ such that if $\f\in L^2(\re^n)$, 
$(P-z)\f\in \mathcal{S}(\re^n)$ and $\i_M(x)\f\in C^\infty(\re^n)$, 
then $\f\in H^{k+1-m/2,-\c}\cap H^{k+1/2,-\c-1/2}$ for any $k\in \re$.
\end{thm}

Now we show Proposition \ref{reg} follows from Theorem \ref{rad}.

\begin{proof}[Proof of Proposition \ref{reg}]
Suppose that $u\in L^2(\re^n)$ satisfies $(P-i)u=0$. By Proposition \ref{locre}, we have $u\in C^{\infty}(\re^n)$. In particular, we have $\i_M(x)\f\in C^\infty(\re^n)$ for any $M\geq 1$. Taking $\c=1/2$ and $k=m-1$, we obtain $\f\in H^{m/2,-1/2}\subset H^{(m-1)/2,-1/2}$. This completes the proof of Proposition \ref{reg}.
\end{proof}

Thus it remains to prove Theorem \ref{rad}.
In the following, we assume $\Im z>0$ without loss of generality. 
We may also assume $0<\c<\min(1/4,\m/2)$. 

\subsection*{Weight functions} 
We choose $\rho(t)\in C^\infty(\re)$ such that 
\[
\rho(t)=\begin{cases} 0 \ &\text{if }t\leq 0, \\ 1 \ &\text{if }t\geq 1/8, \end{cases}
\quad 0\leq \rho(t)\leq 1, \quad \rho'(t)\geq 0 \ \text{ for }t\in\re.
\]
For $\d\in (1/2,7/8)$, we set
\[
\rho_+^\d(t)= \rho(t-\d), 
\quad \rho_-^\d(t)=1-\rho(t+1-\d), 
\quad \rho_0^\d(t)= 1-\rho_+^\d(t)-\rho_-^\d(t),
\]
for $t\in\re$. 
We use the following notation:
\[
\hat x=\frac{x}{|x|}, \quad v(\x)=\pa_\x p_0(\x), \quad \hat v(\x)=\frac{v(\x)}{|v(\x)|}, \quad
\y=\y(x,\x)= \hat x\cdot\hat v(\x).
\]
Then we set 
\[
b^\d(x,\x)= \bigpare{\rho_-^\d(\y)|x|^\c +\rho_0^\d(\y)+\rho_+^\d(\y)|x|^{-\c}} e^{-\c\y},
\]
which is defined for $x, \x\in \re^n\setminus\{0\}$. We introduce cut-off functions and set 
\[
b_{M,\n}^\d(x,\x) = b^\d(x,\x)\oi_M(x)\oi_\n(\x), \quad x,\x\in\re^n. 
\]
with $M,\n>0$. We also write 
\begin{align*}
&\O_1(M,\n)=\bigset{(x,\x)}{M\leq |x|\leq 2M, |\x|\geq \n},\\
&\O_2(M,\n)= \bigset{(x,\x)}{|x|\geq M, \n\leq |\x|\leq 2\n}. 
\end{align*}
The next lemma is a key of the proof of Theorem \ref{rad}. 

\begin{lem}\label{radlem}
Let $1/2<\d<\tilde\d<7/8$, $k\in\re$, $0<\tilde M< M$, $0<\tilde\n<\n$, and write 
\[
B=\Op(b^\d_{M,\n}), \quad \tilde B =\Op(b^{\tilde\d}_{\tilde M,\tilde\n}). 
\]
If $\tilde M$ is sufficiently large, then: 
There are pseudodifferential operators $S=\Op(f_1)$, $T=\Op(f_2)$ such that
$f_1, f_2\in S(1,g)$ and $\supp[f_1]\subset \O_1(M,\n)$, $\supp[f_2]\subset \O_2(M,\n)$; 
If  $\f\in \mathcal{S}'$, $\tilde B\f\in H^{k-1+m/2,-1/2}$, $B(P-z)\f\in H^{k-(m-1)/2,1/2}$, 
$S\f\in H^{k+(m-1)/2}$ and $T\f\in L^2$  then 
\begin{align*}
B\f\in H^k\cap H^{k+(m-1)/2,-1/2}.
\end{align*}
Moreover, For any $N>0$ and $k\geq 0$ there is $C>0$ such that 
\begin{align}
&\norm{B\f}_{H^{k+(m-1)/2,-1/2}}^2 +(\Im z)\norm{B\f}_{H^k}^2 \nonumber \\
&\quad \leq C\bigpare{\norm{B(P-z)\f}_{H^{k-(m-1)/2,1/2}}^2  
+\norm{\tilde B\f}_{H^{k-1+m/2,-1}}^2 \nonumber \\
&\quad \quad \quad +\norm{S\f}_{H^{k+(m-1)/2}}^2 +\norm{T\f}_{L^2}^2
+\norm{\f}_{H^{-N,-N}}^2}. \label{eq-1}
\end{align}
\end{lem}

\begin{rem}
The constant $C$ in the lemma is independent of $\f$ and $z\in\co\setminus\re$. 
We note we assume $\tilde B\f\in H^{k+(m-1)/2,-1/2}$ for technical reasons, though only the norm 
of $\tilde B\f$ in $H^{k,-1}$ appears in the RHS of \eqref{eq-1}. 
\end{rem}

Theorem \ref{rad} follows from Lemma \ref{radlem}. 

\begin{proof}[Proof of Theorem \ref{rad}]
For $j=0,1,2,\dots$, we choose $\n_{j}$ and $\tilde\n_{j}$ so that 
\[
0<\tilde\n_0<\n_0=\tilde\n_{1} <\n_{1}=\tilde\n_2<\n_2=\cdots <\d_0<\infty
\]
with an arbitrarily fixed $\d_0>0$. We then choose $M_{j}$ and $\tilde M_{j}$ so that 
the claim of Lemma \ref{radlem} holds with $k=j/2$, $M=M_{j}$, $\tilde M=\tilde M_{j}$ and 
\[
0<\tilde M_0<M_0=\tilde M_{1}< M_{1}=\tilde M_2<M_2=\cdots.
\]
We also set $\d_j =(1+2^{-j})/4$ and $\tilde \d_j = \d_{j-1} =(1+2\cdot2^{-j})/4$
for $j=0, 1,2,\dots$. 
We write $B_j=\Op(b^{\d_j}_{M_{j},\n_{j}})$, 
$\tilde B_j=\Op(b^{\tilde\d_j}_{\tilde M_{j},\tilde\n_{j}})=B_{j-1}$. 

Suppose $\f\in L^2$ and $(P-z)\f\in \mathcal{S}(\re^n)$. Then we note 
\[
B_j(P-z)\f\in \mathcal{S}(\re^n).
\]
At first, we have $\tilde B_0\f\in H^{0,-\c}\subset H^{0,-1/2}$. 
By Lemma \ref{radlem} with $k=1-m/2$, we learn $\tilde B_{1}\f=B_0\f \in H^{1-m/2}\cap H^{1/2,-1/2}$, 
provided $S\f\in H^{1/2}$ and $T\f\in L^2$, which are satisfied under the assumptions of Theorem \ref{rad}
(with $M_0\leq M$). 
Then we use Lemma \ref{radlem} again with $k=(3-m)/2$ to learn $\tilde B_2\f=B_1\f\in H^{(3-m)/2}\cap H^{1,-1/2}$. 
Iterating this procedure $2k$-times, we arrive at 
\[
B_{2k}\f\in H^{k+1-m/2}\cap H^{k+1/2,-1/2}. 
\]
Note that conditions $S\f\in H^{k/2+1/2}$ and $T\f\in L^2$ are satisfied since $\chi_M(x)\f\in C^{\infty}(\re^n)$. 
Now we use the first inclusion $B_{2k}\f\in H^{k+1-m/2}$. 
We recall, by the assumption, $\i_{M}\f\in H^{k+1-m/2}$, and this implies 
\[
B_{2k}\f+\i_M(x)\f \in H^{k+1-m/2}. 
\]
Since  
\[
b_{M,\n}+\i_{M}(x) \geq c_0\jap{x}^{-\c}, \quad |\x|\geq 2\n, 
\]
by the elliptic estimates (or the sharp G\aa rding inequality), we have $\f\in H^{k+1-m/2,-\c}$. 
$\f\in H^{k+1/2,-\c-1/2}$ follows from $B_{2k}\f\in H^{k+1/2,-1/2}$ by the same argument. 
\end{proof}

For the proof of Lemma \ref{radlem}, we compute the commutator of $B$ and $P$, and then 
use a commutator inequality. 
We write  $b=b^\d_{M,\n}$, $\tilde b = b^{\tilde\d}_{\tilde M,\tilde\n}$, 
$\rho^\d_*=\rho_*$ and $\tilde\rho_*=\rho^{\tilde\d}_*$, where $*=+,-$, or $0$. 
The following lower bound for the Poisson bracket is crucial in the proof of Lemma \ref{radlem}. 

\begin{lem}\label{radcl}
Let $k, M$ and $\n$ be as in Lemma \ref{radlem}. If $M$ is sufficiently large, there are symbols 
$f_1, f_2\in S(1,g)$ such that $\supp[f_1]\subset\O_1(M,\n)$, 
$\supp[f_2]\subset \O_2(M,\n)$, $f_1,f_2\geq 0$, $f_2\leq C\jap{x}^{-(1+\m-2\c)/2}b$, 
and $\d_4>0$ such that 
\[
\{p,\jap{\x}^{2k}b^2\} \geq \d_4 \frac{|v|}{|x|} \jap{\x}^{2k}b^2 
-\jap{\x}^{2k+m-1}f_1^2 -f_2^2. 
\]
\end{lem}

\begin{proof} We first note 
\[
v\cdot \pa_x\y = v\cdot \frac{\pa \hat x}{\pa x} \hat v 
= |v|\biggjap{\hat v, \biggpare{\frac{E}{|x|}-\frac{x\otimes x}{|x|^3}} \hat v}
=\frac{|v|}{|x|}(1-\y^2),
\]
where $E$ denotes the identity matrix. We also note 
\[
\rho_0' = -\rho_+'-\rho_-', \quad \pa_x |x| =\hat x, \quad 
v\cdot (\pa_x|x|) =|v|\hat v\cdot \hat x =|v|\y.
\]
Using these, we compute: 
\begin{align*}
& \{p_0,b\} = v\cdot\pa_x b \\
&= (v\cdot\pa_x\y)\bigbra{\rho'_- |x|^\c +\rho_0' +\rho_+'|x|^{-\c} 
-\c \bigpare{\rho_-|x|^\c +\rho_0 +\rho_+|x|^{-\c}}}\times \\
&\qquad \times \oi_{M}(x)\oi_{\n}(\x) e^{-\c\y}\\
&\quad + (v\cdot \pa_x |x|)\bigpare{\c\rho_-|x|^{\c-1} -\c \rho_+|x|^{-\c-1}}
\oi_{M}(x)\oi_{\n}(\x) e^{-\c\y}\\
&\quad + (v\cdot\pa_x|x|)\bigpare{\rho_-|x|^\c +\rho_0+\rho_+|x|^{-\c}}
M^{-1}\oi'(|x|/M)\oi_{\n}(\x)e^{-\c\y}\\
&= \frac{|v|}{|x|}(1-\y^2)\bigbra{\rho_-'(|x|^\c-1)+\rho'_+(|x|^{-\c}-1)-\c(\rho_-|x|^\c+\rho_0+\rho_+|x|^{-\c})}
\times \\
&\qquad \times \oi_{M}(x)\oi_{\n}(\x) e^{-\c\y}
+ \c\frac{|v|}{|x|}\bigpare{\y\rho_-|x|^\c-\y\rho_+|x|^{-\c}}\oi_{M}(x)\oi_{\n}(\x)e^{-\c\y}+r_0,
\end{align*}
where 
\[
r_0(x,\x)= |v(\x)|\y(x,\x) b^\d(x,\x)
M^{-1}\oi'(|x|/M)\oi_{\n}(\x), 
\]
which is supported in $\O_1(M,\n)$. 
We may suppose $M\geq 1$, and then 
\[
\rho_-'(|x|^\c-1)\leq 0, \quad \rho_+'(|x|^{-\c}-1)\leq 0
\]
on the support of $b$. We also note
\[
\y\rho_-(\y)\leq (-7/8+\d)\rho_-(\y), \  -\y\rho_+(\y)\leq -\d\rho_+(\y), 
\]
and
\[
\ (1-\y^2)\rho_0(\y)\geq \min(1-(\d-1)^2,1-(\d+1/8)^2)\rho_0(\y). 
\]
We set 
\[
\d_3= \min(7/8-\d, \d, 1-(\d-1)^2,1-(\d+1/8)^2)>0. 
\]
We substitute these inequality to the above formula on $\{p_0,b\}$ to learn 
\begin{align*}
\{p_0,b\} &\leq -\c\d_3 \frac{|v|}{|x|}\bigbra{ \rho_0+\rho_-|x|^{\c} +\rho_+|x|^{-\c}}
\oi_{M}(x)\oi_{\n}(\x)e^{-\c\y} +r_0\\
&\leq -\d_3\c \frac{|v|}{|x|}b (x,\x) +r_0(x,\x). 
\end{align*}
Then we have 
\[
-\{p_0,b^2\} =-2b\{p_0,b\} \geq 2\d_3\c\frac{|v|}{|x|} b^2 +2br_0.
\]
This also implies
\begin{equation}\label{eq-2}
-\{p_0,\jap{\x}^{2k}b^2\} =-2b\{p_0,b\}\jap{\x}^{2k} 
\geq 2\d_3\c\frac{|v|}{|x|} \jap{\x}^{2k}b^2 +2\jap{\x}^{2k}br_0.
\end{equation}

On the other hand, we have $\{q,\jap{\x}^{2k}b^2\}\in S(\jap{x}^{-\m+2\c-1}\jap{\x}^{2k+m-1}, g)$. 
We consider this function in more detail.  
We note, for any $\a,\b\in\ze_+^n$, 
\begin{equation}\label{eq-3}
\bigabs{\pa_x^\a\pa_\x^\b b^\d(x,\x)}\leq C_{\a\b}|x|^{\c-|\a|}|\x|^{-|\b|},
\quad x, \x\neq 0,
\end{equation}
with some $C_{\a\b}>0$. 
We also note 
\begin{align*}
\{q, b\} 
&= \{q,b^\d\}\oi_{M}(x)\oi_{\n}(\x) 
+ b^\d\{q,\oi_{M}(x)\oi_{\n}(\x)\} \\
&= \{q,b^\d\}\oi_{M}(x)\oi_{\n}(\x) +r_1+r_2,
\end{align*}
where 
\[
r_1=b^\d(\pa_\x q)\cdot(\pa_x\oi_{M})\oi_{\n}(\x), \quad
r_2= -b^\d\oi_{M}(x)(\pa_x q)\cdot(\pa_\x\oi_{\n}).
\]
We observe that $r_1$ is supported in $\O_1(M,\n)$, 
and $r_1\in S(\jap{\x}^{m-1},g)$; $r_2$ is supported in  $\O_2(M,\n)$
and $r_2\in S(\jap{x}^{-1-\m+\c},g)$. 
Using \eqref{eq-3}, we have 
\begin{align*}
\bigabs{\{q,b^\d\}\oi_M(x)\oi_{\n}(\x) }
&\leq C\jap{x}^{-\m+\c-1}\jap{\x}^{m-1}\oi_M(x)\oi_\n(\x)\\
&\leq C'M^{-(\m-2\c)} \frac{|v(\x)|}{|x|} b(x,\x)
\end{align*}
with some $C,C'>0$. Hence we learn 
\[
\{q,\jap{\x}^{2k}b^2\}\geq -2C'M^{-(\m-2\c)} \frac{|v(\x)|}{|x|} \jap{\x}^{2k}b^2
+2\jap{\x}^{2k}br_1+2\jap{\x}^{2k}br_2, 
\]
uniformly in $M\geq 1$. Combining this with \eqref{eq-2}, we learn 
\[
\{p,\jap{\x}^{2k}b^2\} \geq (2\d_3\c-2C'M^{-(\m-2\c)})\jap{\x}^{2k}\frac{|v|}{|x|} b^2 
+2\jap{\x}^{2k}b (r_0+r_1+r_2).
\]
We recall $\c<\m/2$. We now choose $M$ so large that $2C'M^{-(\m-2\c)}\leq \d_3\c$, and we obtain 
\[
\{p,\jap{\x}^{2k}b^2\} \geq \d_3\c \frac{|v|}{|x|} \jap{\x}^{2k}b^2 +2\jap{\x}^{2k}b (r_0+r_1+r_2).
\]
We note $\supp[r_0+r_1]\Subset\O_1(M,\n)$ and $r_0+r_1\in S(\jap{\x}^{m-1},g)$, hence we can find 
$f_1\in S(1,g)$, $f_1\geq 0$, $\supp[f_1]\subset\O_1(M,\n)$ such that 
\[
2\jap{\x}^{2k}b(r_0+r_1)\geq -\jap{\x}^{2k+m-1}f_1^2.
\]
Similarly, since $\supp[r_2]\Subset\O_2(M,\n)$, $r_2\in S(\jap{x}^{\c-\m-1},g)$, we can find 
$f_2\in S(1,g)$, $f_2\geq 0$, $\supp[f_2]\subset\O_2(M,\n)$ such that 
\[
2\jap{\x}^{2k} b r_2\geq - f_2^2
\quad\text{and}\quad 
0\leq f_2\leq C \jap{x}^{-(1+\m-2\c)/2} b.
\]
By setting $\d_4=\d_3\c$, we arrive at the conclusion of the lemma. 
\end{proof}

We write 
\[
B=\Op(b), \quad \tilde B =\Op(\tilde b), \quad \L=\jap{D_x}^{(m-1)/2}\jap{x}^{-1/2}. 
\]

\begin{lem}\label{radcom} Under the above assumptions, there are pseudodifferential operators 
$S ,T,U, V$ and a constant $\d_4>0$ such that 
\[
-i[P, B \jap{D_x}^{2k} B] \geq \d_4 B\jap{D_x}^k |\L|^2\jap{D_x}^k B -S^*\jap{D_x}^{2k+m-1}S-T^*T-U -V, 
\]
where 
\begin{enumerate}
\item $S\in \Op S(1,g)$ and the symbol is supported in $\O_1(M,\n)$;
\item $T\in \Op S(1, g)$ and the symbol is supported in $\O_2(M,\n)$;
\item $U=\Op(u)$ with $u\in S(\jap{x}^{2\c-2}\jap{\x}^{2k+m-2},g)$ and or any $\a,\b\in\ze_+^n$, 
\[
|\pa_{x}^{\a}\pa_{\x}^{\b}u(x,\x)|\leq C\jap{x}^{-2-|\a|}\jap{\x}^{2k+m-2-|\b|} \tilde b(x,\x)^2; 
\]
\item $V\in \Op  S(\jap{x}^{-\infty}\jap{\x}^{-\infty},g)$.
\end{enumerate}
\end{lem}

In the proof of Lemma~\ref{radcom}, we use the following estimate: 

\begin{lem}\label{radell}
Suppose $a$ be a symbol such that $\supp[a]\subset\O'$, 
where $\O'=\bigset{(x,\x)}{|x|\geq M',|\x|\geq \n'}$ with 
$M'>\tilde M$, $\n'>\tilde\n$, and for any $\a,\b\in\ze_+^n$, 
\[
\bigabs{\pa_x^\a\pa_\x^\b a(x,\x)}\leq C_{\a\b}\jap{x}^{2\ell-|\a|}\jap{\x}^{2s-|\b|}\tilde b(x,\x)^2, 
\]
where $s,\ell\in\re$. Then for any $N$, there is $C,C_N>0$ such that 
\[
|\jap{\f,\Op(a)\f}|\leq C \norm{\tilde B\f}^2_{H^{s,\ell}}+C_N\norm{\f}^2_{H^{-N,-N}}, 
\quad \f\in\mathcal{S}(\re^n).
\]
\end{lem}

\begin{proof}
We note, for any $\a,\b\in\ze_+^n$, 
\[
\bigabs{\pa_x^\a\pa_\x^\b \tilde b(x,\x)}\leq C'_{\a\b}\jap{x}^{-|\a|+2\c(|\a+\b|)}\jap{\x}^{-|\b|}\tilde b(x,\x), 
\quad(x,\x)\in\O'. 
\]
We write $\tilde g= \jap{x}^{-2+4\c}dx^2+\jap{x}^{4\c}\jap{\x}^{-2}d\x^2$. 
Using the above estimate and the assumption on $a$, and following the construction of parametrices for elliptic operators, 
we can construct a symbol $h(x,\x)\in S(1,\tilde g)$ such that 
\[
\Op(a) = \tilde B \jap{x}^\ell \jap{D_x}^s \Op(h) \jap{D_x}^s \jap{x}^\ell \tilde B +R,
\]
where $R\in S(\jap{x}^{-\infty}\jap{\x}^{-\infty},\tilde g)$. The assertion follows from this since $\Op(h)$
is bounded in $L^2(\re^n)$. 
\end{proof}

\begin{proof}[Proof of Lemma \ref{radcom}] 
By the standard pseudodifferential operator calculus, we can find $\tilde f_1$, $\tilde f_2$ 
such that $\tilde f_j\in S(1,g)$, $\supp[\tilde f_j]\subset \O_j(M,\n)$, $j=1,2$, and 
\begin{align*}
&\Op(\jap{\x}^{2k+m-1}f_1^2)\leq \Op(\tilde f_1)^*\jap{D_x}^{2k+m-1}\Op(\tilde f_1) +R_1, \\
&\Op(\jap{x}^{2\c-1-\m}f_2^2)\leq \Op(\tilde f_2)^*\Op(\tilde f_2) +R_2,
\end{align*}
where $R_j$ are smoothing operators. We set $S=\Op(\tilde f_1)$ and $T=\Op(\tilde f_2)$. 
If we write
\[
\z(x,\x)= \{p,\jap{\x}^{2k}b^2\} - \d_4 \frac{|v|}{|x|} \jap{\x}^{2k}b^2 
+\jap{\x}^{2k+m-1}f_1^2 +f_2^2\geq 0. 
\]
We note, by the construction, $\z(x,\x) b'(x,\x)^{-2}\in S(\jap{x}^{-1}\jap{\x}^{2k+m-1},g)$,
where $b'=b_{M',\n'}$ with $\tilde M<M'<M$, $\tilde\n<\n'<\n$. 
Hence by the sharp G\aa rding inequality, we have 
\[
\Op(\z(b')^{-2})\geq -C\jap{D_x}^{k-1+m/2}\jap{x}^{-2}\jap{D_x}^{k-1+m/2} 
\]
with some $C>0$. Then by the asymptotic expansion, we learn 
\[
\Op(\z) \geq -CB' \jap{D_x}^{k-1+m/2}\jap{x}^{-2}\jap{D_x}^{k-1+m/2} B' -R_3, 
\]
where $R_3\in S(\jap{x}^{-3}\jap{\x}^{2k+m-3},g)$, and the symbol is supported in $\supp[b']$
modulo $\mathcal{S}(\re^{2d})$. 
Using Lemma \ref{radcl}, we can estimate $R_3$ and other error terms from below by 
$-C\tilde B \jap{D_x}^{k-1+m/2}\jap{x}^{-2}\jap{D_x}^{k-1+m/2}\tilde B$, 
modulo smoothing operators, and these will be included in $U$ to complete the proof. 
\end{proof}

\begin{lem}\label{radres}
For $\f\in \mathcal{S}(\re^n)$, the inequality \eqref{eq-1} holds, 
where $S=\Op(f_1)$, $T=\Op(f_2)$, $f_1, f_2\in S(1,g)$, and $\supp[f_1]\subset \O_1(M,\n)$, 
$\supp[f_2]\subset \O_2(M,\n)$. 
\end{lem}

\begin{proof}
We compute the commutator to obtain quadratic inequalities. 
For $\f\in\mathcal{S}(\re^n)$, we have 
\begin{align*}
&\bigjap{\f,-i[P,B\jap{D_x}^{2k}B]\f} = \bigjap{\f,-i[(P-z),B\jap{D_x}^{2k}B]\f} \\
&= -i \bigpare{\bigjap{\jap{D_x}^k B(P-\bar z)\f,\jap{D_x}^k B\f}
-\bigjap{\jap{D_x}^k B\f, \jap{D_x}^k B (P-z)\f}} \\
&= -i \bigpare{\bigjap{(\L^{-1})^*\jap{D_x}^k B(P-z)\f,\L\jap{D_x}^k B\f}\\
&\qquad\quad  -\bigjap{\L\jap{D_x}^k B\f, (\L^{-1})^*\jap{D_x}^k B (P-z)\f}} 
-2(\Im z)\bignorm{\jap{D_x}^k B\f}^2\\
&\leq 2\bignorm{(\L^{-1})^* \jap{D_x}^k B (P-z)\f}\cdot\bignorm{\L \jap{D_x}^k B\f}
-2(\Im z)\bignorm{\jap{D_x}^k B\f}^2. 
\end{align*}
Combining this with Lemma \ref{radcom}, we have 
\begin{align*}
&\d_4\bignorm{\L\jap{D_x}^k B\f}^2 +2(\Im z)\bignorm{\jap{D_x}^k B\f}^2\\
&\qquad - \jap{\f,(S^*\jap{D_x}^{2k+m-1}S+T^*T+U+V)\f} \\
&\quad \leq 2\bignorm{(\L^{-1})^* \jap{D_x}^k B (P-z)\f}\cdot\bignorm{\L \jap{D_x}^k B\f}\\
&\quad \leq \frac{\d_4}{2}\bignorm{\L \jap{D_x}^k B\f}^2
+\frac{4}{\d_4} \bignorm{(\L^{-1})^* \jap{D_x}^k B (P-z)\f}^2. 
\end{align*}
Thus we have 
\begin{align*}
&\frac{\d_4}{2}\bignorm{\L \jap{D_x}^k B\f}^2 + 2(\Im z)\bignorm{\jap{D_x}^k B\f}^2\\
& \leq \frac{4}{\d_4} \bignorm{(\L^{-1})^* \jap{D_x}^k B (P-z)\f}^2\\
&\qquad +\jap{\f,(S^*\jap{D_x}^{2k+m-1}S+T^*T+U+V)\f}.
\end{align*}
Now we note, by Lemma \ref{radell}, 
\[
\jap{\f,U\f} \leq C\norm{\tilde B\f}_{H^{k-1+m/2,-1}}^2 +C\norm{\f}_{H^{-N,-N}}^2
\]
with any $N$. These imply \eqref{eq-1} for $\f\in\mathcal{S}(\re^n)$. 
\end{proof}

We now extend Lemma \ref{radres} to more general $\f$ to prove Lemma \ref{radlem}. 
We choose $M'$ and $\n'$ so that $\tilde M<M'<M$, $\tilde \n<\n'<\n$, 
$\d<\d'<\tilde\d$, and set 
\[
b'(x,\x)=b^{\d'}_{M',\n'}(x,\x), \quad B'=\Op(b'). 
\]
We write 
\[
A_\e= \jap{\e D_x}^{-1} B, \quad \tilde A_\e= \jap{\e D_x}^{-1} \tilde B, \quad
A'_\e= \jap{\e D_x}^{-1} B'
\]
and we denote their symbols by $a_\e$, $\tilde a_\e$ and $a'_\e$, respectively. 

By the same computation as in the proof of Lemma \ref{radcl}, we have 
\[
\{p,\jap{\x}^{2k}|a_\e|^2\} \geq \d_4 \frac{|v|}{|x|} \jap{\x}^{2k}|a_\e|^2 
-\jap{\x}^{2k+m-1}f_{1}^2 -\jap{x}^{2\c-1-\m}f_{2}^2, 
\]
modulo $\mathcal{S}(\re^n)$-terms, 
where constants are independent of $\e$, and $f_1$ and $f_2$ are independent of $\e$. 
Then, as well as Lemma \ref{radcom}, we have 
\begin{align*}
&-i[P, A_\e^* \jap{D_x}^{2k} A_\e] \\
&\qquad \geq \d_4 A_\e^* \jap{D_x}^k \L^2\jap{D_x}^k A_\e 
-S^*\jap{D_x}^{2k+m-1}S-T^*T-U_\e -V_\e, 
\end{align*}
where the symbol of $U_\e$ has the property:
\begin{equation}\label{eq-4}
|u_\e(x,\x)|\leq C\jap{x}^{-2}\jap{\x}^{2k+m-2}|a'_\e(x,\x)|^2, 
\end{equation}
and symbols of $U_\e$ and $V_\e$ are bounded in the respective symbol classes. 
It follows that 
\[
|\jap{\f,U_\e\f}|\leq C \norm{A'_\e\f}^2_{H^{k-1+m/2,-1}} +C\norm{\f}^2_{H^{-N,-N}}, \quad \f\in\mathcal{S}(\re^n), 
\]
where the constant is independent of $\e$. 
Thus we have, as well as Lemma \ref{radres}, for $\f\in\mathcal{S}(\re^n)$, 
\begin{align}
&\norm{A_\e\f}_{H^{k+(m-1)/2,-1/2}}^2 +(\Im z)\norm{A_\e\f}_{H^k}^2 \nonumber \\
&\quad \leq C\bigpare{\norm{A_\e(P-z)\f}_{H^{k-(m-1)/2,1/2}}^2  
+\norm{A'_\e\f}_{H^{k-1+m/2,-1}}^2 \nonumber \\
&\quad \quad \quad +\norm{S\f}_{H^{k+(m-1)/2}}^2 +\norm{T\f}_{L^2}^2}
+C_N \norm{\f}_{H^{-N,-N}}^2, \label{eq-5}
\end{align}
with any $N$, where $C$ and $C_N$ are independent of $\e\in(0,1]$. 

\begin{lem}
Suppose that $\f\in\mathcal{S}'(\re^n)$ satisfies $\tilde B\f\in H^{k-1+m/2,-1/2}$, 
\begin{align*}
A_\e(P-z)\f\in H^{k-(m-1)/2,1/2},\, S\f\in H^{k+(m-1)/2}\,\, \text{and}\,\, T\f\in L^2. 
\end{align*}
Then $A_\e\f\in H^{k+(m-1)/2,-1/2}\cap H^m$ and \eqref{eq-5} holds. 
\end{lem}

\begin{proof}
We set, for $L\gg 0$, 
\[
X_L =\i_L(x)\i_L(D_x). 
\]
We first note $\norm{X_L\g-\g}_{H^{s,\ell}}\to 0$ as $L\to\infty$, provided $\g\in H^{s,\ell}$. 
We also note $\g\in H^{s,\ell}$ if and only if $\lim_{L\to\infty}\norm{X_L\g}_{H^{s,\ell}}<\infty$. 

We observe that the symbol of $[X_L,A_\e]$ is bounded by $C\jap{x}^{-1}\jap{\x}^{-1}a_\e'(x,\x)$, 
modulo $\mathcal{S}(\re^{2d})$-terms, uniformly in $L$, and also it converges to 0 locally 
uniformly as $L\to\infty$. These imply 
\begin{align*}
\lim_{L\to\infty} \norm{X_L A_\e\g}_{H^{s,\ell}}
&\leq \varliminf_{L\to\infty} ( \norm{A_\e X_L \g}_{H^{s,\ell}} + \norm{[X_L,A_\e]\g}_{H^{s,\ell}})\\
&\leq \varliminf_{L\to\infty} \norm{A_\e X_L \g}_{H^{s,\ell}} 
\end{align*}
with any $N$, provided $\tilde{B}\g\in H^{s-1,\ell-1}$. In particular, since we assume $\tilde{B}\f\in H^{k-1+m/2,-1/2}$, 
\begin{multline*}
\lim_{L\to\infty} (\norm{X_L A_\e\f}_{H^{k+(m-1)/2,-1/2}}^2 +\norm{X_L A_\e\f}_{H^k}^2)\\
\leq \varliminf_{L\to\infty}(\norm{A_\e X_L \f}_{H^{k+(m-1)/2,-1/2}}^2 
+\norm{A_\e X_L \f}_{H^k}^2).
\end{multline*}
By the same argument, using $\tilde{B}\f\in H^{k-1+m/2,-1/2}$, we learn 
\[
\varlimsup_{L\to\infty} \norm{A_\e (P-z) X_L\f}_{H^{k-(m-1)/2,1/2}}^2
\leq \norm{A_\e (P-z)\f}_{H^{k-(m-1)/2,1/2}}^2. 
\]
We have similar estimates for $\norm{S\f}_{H^{k+(m-1)/2}}$ and $\norm{T\f}_{L^2}$. 
Concerning the estimate for $\norm{A_\e'\f}_{H^{k-1+m/2,-1}}$, we use the fact that 
$\tilde B\f\in H^{k-1+m/2,-1/2}$ to obtain 
\[
\varlimsup_{L\to\infty} \norm{A_\e' X_L\f}_{H^{k-1+m/2,-1}}^2
\leq \norm{A'_\e\f}_{H^{k-1+m/2,-1}}^2. 
\]
Combining these with \eqref{eq-5} for $X_L\f$, we learn 
\begin{align*}
& \lim_{L\to\infty} (\norm{X_L A_\e\f}_{H^{k+(m-1)/2,-1/2}}^2 +\norm{X_L A_\e\f}_{H^k}^2)\\
&\quad \leq  \varlimsup_{L\to\infty} \bigpare{ C(\norm{A_\e (P-z) X_L\f}_{H^{k-(m-1)/2,1/2}}^2
+\norm{A_\e'X_L \f}_{H^{k-1+m/2,-1}}^2 \\
&\quad\quad\quad   + \norm{SX_L \f}_{H^{k+(m-1)/2}}^2 +\norm{TX_L \f}_{L^2}^2)
+C_N \norm{X_L \f}_{H^{-N,-N}}^2} \\
&\quad \leq C\bigpare{\norm{A_\e(P-z)\f}_{H^{k-(m-1)/2,1/2}}^2  
+\norm{A'_\e\f}_{H^{k-1+m/2,-1}}^2  \\
&\quad \quad\quad+\norm{S\f}_{H^{k+(m-1)/2}}^2 +\norm{T\f}_{L^2}^2}
+C'_N \norm{\f}_{H^{-N,-N}}^2,
\end{align*}
and this implies the assertion. 
\end{proof}

\begin{proof}[Proof of Lemma \ref{radlem}]
It remains to take the limit $\e\to 0$ in \eqref{eq-5}. 
We note
\begin{align*}
\norm{A_\e \f}_{H^{s,\ell}} &= 
\norm{\jap{D_x}^s\jap{x}^\ell \jap{\e D_x}^{-1}B\f}_{L^2} \\
&= \norm{\jap{\e D_x}^{-1}\jap{D_x}^s\jap{x}^\ell B\f+ 
\jap{D_x}^s [\jap{x}^\ell, \jap{\e D_x}^{-1}]B\f}_{L^2},
\end{align*}
and hence 
\[
\norm{\jap{\e D_x}^{-1}\jap{D_x}^s\jap{x}^\ell B\f}_{L^2}
\leq \norm{A_\e\f}_{H^{s,\ell}} + C\norm{B\f}_{H^{s-1,\ell-1}}.
\]
Thus we have 
\[
\norm{B\f}_{H^{s,\ell}}\leq \varliminf_{\e\to 0} \norm{A_\e\f}_{H^{s,\ell}}
+C\norm{B\f}_{H^{s-1,\ell-1}}.
\]
We note this holds without assuming $B\f\in H^{s,\ell}$, and if the right hand side 
is finite, we obtain $B\f\in H^{s,\ell}$. 

By the same argument, we also have 
\begin{align*}
&\varlimsup_{\e\to 0} \norm{A_\e(P-z)\f}_{H^{k-(m-1)/2,1/2} }\\
&\quad \leq \norm{B(P-z)\f}_{H^{k-(m-1)/2,1/2}} + C\norm{B(P-z)\f}_{H^{k-(m-3)/2,-1/2} }\\
&\quad \leq (1+C) \norm{B(P-z)\f}_{H^{k-(m-1)/2,1/2}}
\end{align*}
and similarly, 
\[
\varlimsup_{\e\to 0} \norm{A_\e'\f}_{H^{k-1+m/2,-1} }
\leq C' \norm{B'\f}_{H^{k-1+m/2,-1} }. 
\]
Substituting these to \eqref{eq-5}, we have 
\begin{align*}
&\norm{B\f}^2_{H^{k+(m-1)/2,-1/2}}+(\Im z)\norm{B\f}^2_{H^{k}} \\
&\quad \leq \varliminf_{\e\to 0} \bigpare{\norm{A_\e\f}^2_{H^{k+(m-1)/2,-1/2}}
+C\norm{B\f}^2_{H^{k+(m-3)/2,-3/2}}\\
&\quad\quad\quad   + (\Im z)(\norm{A_\e\f}^2_{H^{k}}+C\norm{B\f}^2_{H^{k-1,-1}})}\\
&\quad \leq \varlimsup_{\e\to 0}C \bigpare{ \norm{A_\e(P-z)\f}^2_{H^{k-(m-1)/2,1/2}}
+\norm{\tilde A_\e \f}^2_{H^{k,-1}}+\norm{B\f}^2_{H^{k-1/2,-1}}} \\
&\quad\quad\quad  +C(\norm{S\f}^2_{H^{k+(m-1)/2}}+\norm{T\f}_{L^2}^2)
+C_N \norm{\f}_{H^{-N,-N}}^2\\
&\quad \leq C'\bigpare{\norm{B(P-z)\f}_{H^{k-(m-1)/2,1/2}}^2  
+\norm{\tilde B\f}_{H^{k,-1}}^2 \nonumber \\
&\quad \quad \quad +\norm{S\f}_{H^{k+(m-1)/2}}^2 +\norm{T\f}_{L^2}^2}
+C_N \norm{\f}_{H^{-N,-N}}^2,
\end{align*}
and this completes the proof of Lemma \ref{radlem}. 
\end{proof}

\appendix
\section{Estimates for the classical trajectories}\label{appa}
In this section, we prove estimates on the classical trajectories which are used in the proof of Proposition \ref{singprop}.
First, we show a classical Mourre estimate which implies the peudo-convexity of $\re^n$ with respect to $P$. We note
\[
(y(t,x,\l\x),\y(t,x,\l\x))=(y(\l^{m-1}t,x,\x),\l\y(\l^{m-1}t,x,\x))\quad \text{for }\l>0, 
\]
since $p_m$ is homogeneous of degree $m$. 
\begin{lem}\label{mourre}
There exist $M>0$ and $R_0>1$ such that 
\[
H_{p_m}^2(|x|^2)\geq M|\x|^{2(m-1)}
\]
for any $(x,\x)\in \{(y,\y)\in T^*\re^{n}\,|\, |y|>R_0,\,|\y|\neq 0\}$. 
\end{lem}

\begin{proof}
We have 
\begin{align*}
&H_{p_m}^2(|x|^2)=2H_{p_m}(x\cdot \pa_{\x}p_m)\\
&\quad =2|\pa_{\x}p_m|^2+2\sum_{j,k=1}^{n}x_j(\pa_{x_k}\pa_{\x_j}p_m)\pa_{\x_k}p_m-
2\sum_{j,k=1}^nx_j(\pa_{\x_j}\pa_{\x_k}p_m)\pa_{x_k}p_m.
\end{align*}
On the other hand, by Assumption~\ref{assa}, there exists $C>0$ such that
\[
\biggabs{2\sum_{j,k=1}^{n}x_j(\pa_{x_k}\pa_{\x_j}p_m)\pa_{\x_k}p_m-
2\sum_{j,k=1}^nx_j(\pa_{\x_j}\pa_{\x_k}p_m)\pa_{x_k}p_m}
\leq C\jap{R_0}^{-\m}|\x|^{2(m-1)}.
\]
Combining this with the non-degeneracy condition of $\pa_\x p_0(\x)$ in Assumption~\ref{assa}, 
we conclude the assertion. 
\end{proof}

Next, we observe that an energy bound on classical trajectories holds, even if $p$ is not elliptic. 
We note an analogous result is proved in \cite{KPRV}, though our proof is simpler. 

\begin{lem}\label{energy}
Fix $(x_0,\x_0)\in T^{*}\re^n$ with $\x_0\neq 0$ and suppose that $(x_0,\x_0)$ is forward non-trapping in the sense that $|y(t,x_0,\x_0)|\to \infty$ as $t\to \infty$. Then, there exist $C_1, C_2>0$ such that
\[
C_1\leq |\y(t, x_0, \x_0)|^{m-1}\leq C_2,
\]
for $t\geq 0$.
\end{lem}
\begin{proof}
Let $R_0$ be as in Lemma~\ref{mourre}, and we let $R_1\geq R_0$ which is determined later.
We first note that by the forward non-trapping condition and Lemma \ref{mourre}, there exits $t_0\geq 0$ such that for $t\geq t_0$, we have
\begin{equation}\label{5.1}
|y(t,x_0,\x_0)|\geq R_1,\quad \frac{d}{dt}|y(t,x_0,\x_0)|^2\geq 0.
\end{equation}
By Lemma~\ref{mourre} and the non-trapping condition, it is easy to observe that there is $t_0>0$ such that 
$\frac{d^2}{dt^2}|y(t,x_0,\x_0)|^2>0$ for $t\geq s_0$, and $\frac{d}{dt}|y(t_0,x_0,\x_0)|^2> 0$. 
Then for all $t\geq t_0$, the condition \eqref{5.1} is satisfied. 

Let $C_0>0$ be a constant such that
\[
|\pa_{x}p_m(x,\x)|\leq C_0 |x|^{-1-\m} |\x|^{m},
\]
and we write $ \y_0=|\y(t_0,x_0,\x_0)|>0$. 
We set 
\[
T=\sup\bigset{s\geq t_0}{\y_0/2\leq |\y(t,x_0,\x_0)|\text{ for }t\in [t_0,s]}\in (t_0,\infty].
\]
By Lemma~\ref{mourre}, we have 
\[
|y(t,x_0,\x_0)|^2\geq R_1^2+\frac{M\y_0^{2(m-1)}}{2^{2m-1}}(t-t_0)^2, \quad t_0\leq t\leq T. 
\]
Now we note
\[
\left|\frac{d}{dt}|\y(t,x_0,\x_0)|\right|\leq C_0 |y(t,x_0,\x_0)|^{-1-\m}|\y(t,x_0,\x_0)|^m
\]
and hence
\[
\left|\frac{d}{dt}|\y(t,x_0,\x_0)|^{-(m-1)}\right|
\leq (m-1)C_0 \biggpare{R_1^2+\frac{M\y_0^{2(m-1)}}{2^{2m-1}}(t-t_0)^2}^{-(1+\m)/2}
\]
for $t_0\leq t\leq T$. Thus we have 
\begin{align*}
&\left| \y_0^{-(m-1)}-|\y(T,x_0,\x_0)|^{-(m-1)}\right| \\
&\quad \leq \int_{t_0}^T (m-1)C_0 \biggpare{R_1^2+\frac{M\y_0^{2(m-1)}}{2^{2m-1}}(t-t_0)^2}^{-(1+\m)/2}dt \\
&\quad \leq \frac{C_02^{(2m-1)/2}R_1^{-\m}}{(1+\m)\sqrt{M}}\y_0^{-(m-1)}.
\end{align*}
We now choose $R_1>0$ so large that 
\[
\frac{C_02^{(2m-1)/2}R_1^{-\m}}{(1+\m)\sqrt{M}}<1/2, 
\quad\text{i.e.,}\quad
R_1> \biggpare{\frac{C_0 2^{(2m+1)/2}}{(1+\m)\sqrt{M}}}^{1/\m}, 
\]
then 
\[
|\y(T,x_0,\x_0)|^{-(m-1)}<(3/2)\y_0^{-(m-1)}, 
\]
i.e., $|\y(T,x_0,\x_0)|>(2/3)^{1/(m-1)}\y_0>(1/2)\y_0$. 
If $T<\infty$, this is a contradiction, and hence $T=\infty$. Thus we also learn
\[
2^{-1}\y_0\leq |\y(t,x_0,\x_0)|\leq 2^{1/(m-1)}\y_0, \quad t\geq t_0.
\]
\end{proof}

\begin{cor}
Suppose the same assumptions as in Lemma \ref{energy} hold. Moreover, suppose $|\x_0|=1$. Then, we have 
\begin{align*}
C_1\l\leq |\y(t,x_0,\l\x_0)|\leq C_2\l
\end{align*}
for any $\l>0$ and $t\geq 0$.
\end{cor}

\begin{cor}
Under the same assumptions as in Lemma \ref{energy} with $|\x_0|=1$, there exist $C,C',K,K'>0$ such that
\begin{align*}
C\l t -K\leq |y(t,x_0,\x_0)|\leq C'\l t+K'
\end{align*}
for $\l>0$ and $t\geq 0$.
\end{cor}

Combining with the estimate $|\pa_{x}p_m(x,\x)|\leq C\jap{x}^{-1-\m}|\x|^{m-1}$, we obtain:
\begin{cor}\label{asymmom}
Suppose that $(x_0,\x_0)\in \re^n\times \re^n\setminus\{0\}$ is non-trapping. Then, 
\begin{align*}
\y_{+}&=\lim_{t\to\infty}\y(t,x_0,\x_0)\neq 0,\\
v_{+}&=\lim_{t\to\infty}\pa_{\x}p_m(y(t,x_0,\x_0),\y(t,x_0,\x_0))=\lim_{t\to\infty}\pa_{\x}p_0(\y(t,x_0,\x_0))\neq 0
\end{align*}
exist.
\end{cor}

\section{Construction of the conjugate operator}\label{appb}

Let $(x_0,\x_0)\in p_m^{-1}(0)\setminus\{\x\neq 0\}$. By Assumption \ref{assnull}, $(x_0,\x_0)$ is forward non-trapping.
We denote $y(t) = y(t, x_0, \x_0)$, $\y(t) = \y(t, x_0, \x_0)$. 
We note that
\[
\lim_{j\to\infty}\y(t,x_0,\x_0)=\y_+\neq 0,\,\, \lim_{t\to \infty} \pa_{\x}p_m(y(t),\y(t)) = v_+\neq 0,
\]
exist by Corollary \ref{asymmom}. Moreover, there exist $M_1, M_2>0$ such that
\begin{equation}\label{cl}
\begin{aligned}
&|y(t)/t-v_+|,\,\, |\y(t)-\y_+|=O(\jap{t}^{-\m})\quad\text{as } t\to\infty,\\
& M_1\leq|\y(t)|\leq M_2,\quad t\geq 0.
\end{aligned}
\end{equation}

We denote $B(r,s,z,\z) = \{(x,\x)\in \re^{2n}| |z-x| < r, |\z -\x| < s\}\subset \re^{2n}$. 
In order to prove Proposition \ref{singprop}, it suffices to prove the following theorem.
We set an $h$-dependent metric $g_h$ by
\[
g_h = dx^2/\jap{x}^2+h^{2/(m-1)}d\x^2.
\]

\begin{thm}\label{sing}
There exist $\g_{h}\in C_c^\infty(\re^{2n})$ and $\f_{h,t} \in C^{\infty}(\re_{\geq 0},C_c^{\infty}(\re^{2n}))$ such that $F(h,t) = \Op(\f_{h,t})$ and:
\begin{enumerate}
\item $F(h,0) =|\Op(\g_{h})|^2\;\; \text{with}\;\; \g_{h}(x_0,\x_0) \geq 1$.
\item $\f_{h,t}$ satisfies
\[
\supp \f_{h,t}\subset B(4h^{-1}t\d_1,4h^{-1/(m-1)}\d_2,h^{-1}tv_+,h^{-\frac{1}{m-1}}\y_+)
\]
modulo $S(h^{\infty},g_h)$ if $t/h$ is sufficiently large.
\item For any $\a,\b\in \mathbb{N}_{\geq0}^n$, there exists $C_{\a\b}>0$ such that
\[
|\pa_x^{\a}\pa_{\x}^{\b}\f_{h,t}(x,\x)|\leq C_{\a\b}\jap{t} h^{(|\b|+1)/(m-1)-1} \jap{x}^{-|\a|}.
\]
\item There exists a family of bounded operator $R(h,t)$ in $L^2(\re^n)$ such that
\[
\frac{\pa F}{\pa t} + i[P, F] \geq -R(h,t),
\]
where $\sup_{\geq 0} \jap{t}^{-1}\|R(h,t)\|_{L^2\rightarrow L^2} = O(h^{\infty})$.
\end{enumerate}
\end{thm}

The proof of Theorem \ref{sing} is based on the fact that any classical trajectory of $H_p$ behave as straight lines even if $p$ is not elliptic. 
We follow the argument in \cite{N}.

\begin{lem}\label{sym}
There exist constants $\d_1, \d_2 >0$ with $|\y_+|>4\d_1$ such that the following holds:

There exists a smooth function $\g \in C^{\infty}(\re_{\geq0}, C_c^{\infty}(\re^{2n}))$ such that 
\begin{enumerate}
\item $\g \geq 0$, and $\g(0, x_0, \x_0) \geq 1$.
\item $\supp \g(t, \cdot, \cdot) \subset B(2t\d_1,2\d_2,tv_+,\y_+)$ for $t\geq T_0$, where $T_0>0$ depends only on $(x_0,\x_0)$, $p_m$ and $\d_1$.
\item For any $\a,\b \in \mathbb{N}^n$, there exists $C_{\a\b}>0$ such that 
\[
|\pa_x^{\a}\pa_{\x}^{\b}\g(t,x,\x)|\leq C_{\a\b}\jap{x}^{-|\a|},\,\,|\pa_x^{\a}\pa_{\x}^{\b}\pa_t\g(t,x,\x)|\leq C_{\a\b}\jap{x}^{-1-|\a|}
\]
for $t\geq 0$ and $x,\x \in \re^n$.
\item $\g$ satisfies
\[
\biggpare{\frac{\pa \g}{\pa t} + \{p_m, \g \}}(t, x,\x) \geq 0
\]
for $t\geq 0$, $x,\x \in \re^n$.
\end{enumerate}
\end{lem}

\begin{proof}
Let $\Psi \in C^{\infty}(\re)$ such that $0\leq \Psi \leq 1$, $\Psi'\leq0$, $\Psi =1$ for $r \leq \frac{1}{2}$, $\Psi =0$ for $r\geq1$, $\Psi(r)>0$ if $\frac{1}{2}<r<1$.
We define
\[
\g_0(t,x,\x)\coloneqq \Psi\biggpare{\frac{|x-y(t)|}{\d_1\jap{t}}}\Psi\biggpare{\frac{|\x-\y(t)|}{\c(t)}}
\]
where we set $\c(t)=\d_2-C_1\jap{t}^{-\m}$ and let $C_1>0$ be determined later. 
We set
\begin{align*}
&L(t,x,\x)=\pa_{\x}p_m(x,\x)-\pa_{\x}p_m(y(t),\y(t)),\\
&A_0(t,x,\x)=\frac{1}{\d_1\jap{t}}\biggpare{L(t,x,\x)\cdot \frac{x-y(t)}{|x-y(t)|} -\frac{t|x-y(t)|}{\jap{t}^2}},\\
&A_1(t,x,\x)=\frac{1}{\c(t)}\biggpare{-\frac{\c'(t)|\x-\y(t)|}{\c(t)}+\bigpare{\pa_xp(y(t),\y(t))-\pa_xp(x,\x)}\cdot \frac{\x-\y(t)}{|\x-\y(t)|}}.
\end{align*}
For $t>0$, we have
\begin{align}
\biggpare{\frac{\pa \g_0}{\pa t}+\{p_m, \g_0 \}}(t, x,\x)=&A_0(t,x,\x)\Psi'\biggpare{\frac{|x-y(t)|}{\d_1t}}\Psi\biggpare{\frac{|\x-\y(t)|}{\c(t)}}\label{Psi0} \\
&+A_1(t,x,\x) \Psi\biggpare{\frac{|x-y(t)|}{\d_1t}}\Psi'\biggpare{\frac{|\x-\y(t)|}{\c(t)}}.\nonumber
\end{align}
Using $|\pa_{\x}p(x,\x)-\pa_{\x}p(y(t),\y(t))|\leq C_0|\x-\y(t)|$ with a constant $C>0$, we have
\begin{align}\label{A_0}
\d_1\jap{t}A_0(t,x,\x)\leq -\frac{\d_1 t}{2\jap{t}}+C_0\c(t)\leq -\frac{\d_1 t}{2\jap{t}}+C_0\d_2-C_0C_1\jap{t}^{-\m}
\end{align}
for $(x,\x)\in \supp \Psi'(|x-y(t)|/\d_1\jap{t})\Psi(|\x-\y(t)|/\c(t))$. By Assumption \ref{assa} and $(\ref{cl})$, there exists $C,T_{00}>0$ such that for $(x,\x)\in \supp \g_0(t,x,\x)$, we have
\[
|\pa_{x}p_m(y(t),\y(t))-\pa_{x}p_m(x,\x)|\leq C\jap{t}^{-1-\m}
\]
for $t\geq T_{00}$. Here, we can choose $C, T_{00}>$ independently of $C_1$. We note that  and $\c(t)/2\leq |\x-\y(t)|$ holds on the support of $\G'(|\x-\y(t)|/\c(t))$. Using these observations, we learn 
\begin{align}
&A_1(t,x,\x)\leq -\frac{\c'(t)}{\c(t)^2}|\x-\y(t)|+\frac{C\jap{t}^{-1-\m}}{\c(t)}  \label{A_1}\\
&\quad =\frac{1}{\c(t)}\biggpare{-\frac{C_1\m t}{\jap{t}^{2+\m}}\cdot \frac{|\x-\y(t)|}{\c(t)}+\frac{C}{\jap{t}^{1+\m}}}
\leq-\frac{1}{\c(t)}\biggpare{\frac{C_1\m t}{2\jap{t}^{2+\m}}-\frac{C}{\jap{t}^{1+\m}}}
\nonumber
\end{align}
for $(x,\x)\in \supp \Psi(|x-y(t)|/\d_1\jap{t})\Psi'(|\x-\y(t)|/\c(t))$ with $t\geq T_{00}$. By (\ref{Psi0}), (\ref{A_0}) and (\ref{A_1}) with $\G'\leq 0$ and $\d_1>>\d_2$, we can select  $T_{00}>0$ and $C_1>0$ such that for $t\geq T_{00}$,
\begin{equation}\label{trans0}
\biggpare{\frac{\pa \g_0}{\pa t} + \{p_m, \g_0 \}}(t, x,\x)\geq 0.
\end{equation}

Now we define $\g(t,x,\x)$ by the solution to
\begin{align}
\biggpare{\frac{\pa \g}{\pa t} + \{p_m, \g \}}(t, x,\x)=& \rho(t)\biggpare{\frac{\pa \g_0}{\pa t} + \{p_m, \g_0 \}}(t, x,\x), \quad 0\leq t\leq T_{00}+1,\label{trans}\\
\g(T_{00}+1,x,\x)=& \g_0(T_{00}+1, x, \x),\nonumber
\end{align}
where $\rho \in C^{\infty}(\re, [0,1])$ such that $\rho(t)=1$ for $t\geq T_{00}+1$, $\rho(t) = 0$ for $t\leq T_{00}$. Then we can extend $\g$ smoothly to $t\geq T_{00} + 1$ by $ \g(t,x,\x) = \g_0(t, x, \x)$ for $t \geq T_{00}+1$. For $(x,\x)\in \re^{2n}$, by using $\rho(t)\leq 1$, we obtain
\begin{align*}
\frac{d\g}{dt}(t,y(t,x,\x),\y(t,x,\x))&\leq \frac{d\g_0}{dt}(t,y(t,x,\x),\y(t,x,\x)).
\end{align*}
Let $0\leq s\leq T_{00}+1$. Integrating this inequality over $[s,T_{00}+1]$ with $(x,\x)=(x_0,\x_0)$ and using $ \g(t,x,\x) = \g_0(t, x, \x)$ with $(t, x,\x)=(T_{00}+1, y(T_{00}+1), \y(T_{00}+1))$, we have
\begin{align*}
\g(s,y(s),\y(s)) &\geq \g_0(s,y(s),\y(s))\geq 0.
\end{align*}
Substituting this inequality with $s=0$, we have $\g(0,x_0,\x_0)\geq \g_0(0,x_0,\x_0)=1$.
This implies that $\g$ satisfies (i). We set $T_0=T_{00}+1$. Now (ii) follows from (\ref{cl}) and the relation $\g(t,x,\x)=\g_0(t,x,\x)$ for $t\geq T_0$. (iv) follows from (\ref{trans0}) and (\ref{trans}). Furthermore, (iii) follows from (\ref{cl}), (\ref{trans}), the relation $\g(t,x,\x)=\g_0(t,x,\x)$ for $t\geq T_0$ and the definition of $\g_0$.
\end{proof}

We set
\begin{equation}\label{gdef}
\g_{h,t}(x,\x)= \g(t/h,x,h^{\frac{1}{m-1}}\x),\quad \f_{0,h,t}(x,\x)=\g_{h,t}\#\g_{h,t}(x,\x),
\end{equation}
and $F_0(h,t) =\Op(\f_0(h,t,\cdot, \cdot))=|\Op(\g_{h,t})|^2$,
where $\#$ denotes the composition of the Weyl quantization (\cite[$(4.3.6)$]{Z} with $h=1$) and $|A|^2=A^*A$ for an operator $A$.
\begin{lem}\label{op}
\begin{enumerate}
\item $F_0(0) =|\Op(\g_{h,0})|^2\;\; \text{with}\;\; \g_{h,0}(x_0,\x_0) \geq 1$.
\item We have
\[
\supp \f_{0,h,t}\subset B\bigpare{2h^{-1}t\d_1,2h^{-\frac{1}{m-1}}\d_2,h^{-1}tv_+,h^{-\frac{1}{m-1}}\y_+}
\]
modulo $S(h^{\infty},g_h)$ if $t/h\geq T_1$.
\item For any $\a,\b \in \mathbb{N}_{\geq0}^n$, there exists $C_{\a\b}>0$ such that
\begin{align*}
&|\pa_x^{\a}\pa_{\x}^{\b}\f_{0,h,t}(x,\x)|\leq C_{\a\b} h^{\frac{|\b|}{m-1}} \jap{x}^{-|\a|},\\
&|\pa_x^{\a}\pa_{\x}^{\b}\pa_t \f_{0,h,t}(x,\x)|\leq C_{\a\b} h^{\frac{|\b|}{m-1}-1} \jap{x}^{-|\a|-1}.
\end{align*}
\item There exists $r_0(t,x,\x) \in C^{\infty} (\re_{\geq 0} \times \re^{2n})$ such that
\[
\frac{\pa}{\pa t} F_0(h,t) + i[P, F_0(h,t)] \geq -\Op(r_{0,h,t}),
\]
and $\supp r_{0,h,t}\subset \supp \f_{0,h,t}$ modulo $S\bigpare{h^{\infty}\jap{x}^{-\infty},g_h}$. Moreover, for any $\a,\b \in \mathbb{N}_{\geq0}^n$, there exists $C_{\a\b}>0$ such that
\[
|\pa_x^{\a}\pa_{\x}^{\b}r_{0,h,t}(x,\x)|\leq C_{\a\b} h^{\frac{|\b|-(m-2)}{m-1}} \jap{x}^{-|\a|-1-\m}.
\]
\end{enumerate}
\end{lem}

\begin{proof}
Propeties (i)--(iii) follow from (\ref{cl}) and Lemma \ref{sym}. We prove (iv). 
Since $|x|\sim t/h$ holds on $\supp \g_{h,t}$, we learn
$\pa_t\f_{0,h,t}(\cdot, \cdot)\in S(h^{-1}\jap{x}^{-1} ,g_h)$. 
Moreover, we have
$[P, F_0(h,t)]\in \Op S(\jap{x}^{-1}\jap{\x}^{m-1},  g_h)$. 
By its support property, $[P, F_0(h,t)] \in \Op S(h^{-1}\jap{x}^{-1},  g_h)$ follows. We obtain
\[
\frac{\pa}{\pa t} |\g_{h,t}(h,t,x,\x)|^2 + \{p_m, |\g_{h,t}(\cdot ,\cdot)|^2\}(x,\x)\geq 0
\]
by Lemma \ref{sym} $(iv)$. We note $p=p_m+V$ with $V\in S^{m-1, -\m}$ and 
\[
[V, F_0(h,t)]\in \Op S\bigpare{h^{-\frac{m-2}{m-1}}\jap{x}^{-1-\m}, g_h}.
\]
By the sharp G\aa rding inequality, there exists 
$r_{0,h,t} \in S\bigpare{h^{\frac{-(m-2)}{m-1}} \jap{x}^{-1-\m},g_h}$ such that (iv) holds.
\end{proof}

\begin{proof}[Proof of Theorem \ref{sing}] We choose $\l_0,\l_1,\l_2, . . . \in[1,2)$ such that
\[
1 = \l_0 <\l_1<\l_2<\cdot \cdot \cdot <2,
\]
and take $\g_{k,h,t}(x,\x)$ as $\g_{h,t}(x,\x)$ and $T_k$ as $T_0$ with $\d_j$ replaced by $\l_k\d_j$ in Lemma \ref{sym} and \eqref{gdef}. By the choice of $\Psi$, we note
\begin{equation}\label{gsupp}
\g_{k+1,h,t}(x,\x)\geq L_k
\end{equation} 
on $\supp \g_{k,h,t}(\cdot,\cdot)$ for some $L_k>0$. For $k\geq 1$, set
\[
\f_{k,h,t}(x,\x) = h^{\frac{k-m+1}{m-1}}tC_k\g_{k,h,t}\# \g_{k,h,t} \in S(h^{\frac{k-m+1}{m-1}}t,g_h)
\]
where $C_k>0$ is determined later.
By Lemma \ref{op} (iv), we can write $r_{0,h,t} = r_{01,h,t} + r_{02,h,t}$, where 
\begin{align}\label{r_01}
r_{01,h,t}\in S\bigpare{h^{\frac{-(m-2)}{m-1}}\jap{x}^{-1-\m} ,g_h}
\end{align}
satisfies $\supp r_{01,h,t}(t,\cdot,\cdot) \subset \supp \f_0(t,\cdot,\cdot)$ and $r_{02,h,t} \in S(h^{\infty}\jap{x}^{-\infty},g_h)$. By (\ref{gsupp}), we can find $C_1>0$ such that
\[
r_{01,h,t}(x,\x) \leq C_1h^{\frac{-m+2}{m-1}}|\g_{1,h,t}(x,\x)|^2.
\]
This inequality with Lemma \ref{sym} (iv) implies
\begin{align}
&C_1h^{\frac{-m+2}{m-1}}\biggpare{\frac{\pa}{\pa t} (t|\g_{1,h,t}|^2) +t\{p_m,|\g_{1,h,t}|^2\}}(x,\x)\label{gain2} \\
&\quad = C_1h^{\frac{-m+2}{m-1}}t\biggpare{\frac{\pa}{\pa t} |\g_{1,h,t}|^2+\{p_m,|\g_{1,h,t}|^2\}}(x,\x) + C_1h^{\frac{-m+2}{m-1}}|\g_{1,h,t}(x,\x)|^2\nonumber\\
&\quad \geq r_{01,h,t}(x,\x).\nonumber
\end{align}
Taking $M_k=\max(T_k, ||v_+|-2\l_k\d_1|)>0$, we have
\begin{equation}\label{trange}
t\leq M_kh\jap{x},\,\, \text{for}\,\, (t,x,\x)\in \supp \g_{k,h,t}
\end{equation}
by Lemma \ref{op} (ii). Lemma \ref{op} (iii) with \eqref{trange} implies
\begin{equation}\label{gasym}
h^{\frac{-m+2}{m-1}}t\biggpare{\frac{\pa |\g_{1,h,t}|^2}{\pa t}+\{p_m,|\g_{1,h,t}|^2\}}\in S\bigpare{h^{\frac{-m+2}{m-1}}, g_h)}.
\end{equation}
By \eqref{r_01}, \eqref{gain2} and \eqref{gasym}, it follows that the both sides in \eqref{gain2} belong to $S(h^{\frac{-m+2}{m-1}}, dx^2/\jap{x}^2+h^{2/(m-1)}d\x^2).$
The sharp G\aa rding inequality implies that there exists 
\[
r_{1,h,t} \in S(h^{\frac{-m+3}{m-1}}\jap{x}^{-1}, g_h)
\]
 which is supported in 
$\supp \f_{1,h,t}$ modulo $S(h^{\infty}\jap{x}^{-\infty},g_h)$ such that 
\[
\frac{\pa}{\pa t} \Op(\f_{1,h,t}) + i[P, \Op(\f_{1,h,t})] \geq \Op(r_{0,h,t}) - \Op(r_{1,h,t}).
\]

We set $F_1(h,t) = F_0(h,t) +\Op(\f_{1,h,t})$, then we have
\[
\frac{\pa}{\pa t} F_1(h,t) + i[P,F_1(h,t)] \geq  - \Op(r_{1,t,h}).
\]
Iterating the above argument, we can construct $C_k>0$, $F_k(t)$ and 
\[
r_{k,h,t} \in S\bigpare{h^{\frac{k-m+2}{m-1}}\jap{x}^{-1},g_h}
\]
such that $\supp r_{k,h,t}\subset \supp \f_{k,h,t}(\cdot,\cdot)$ modulo $S(h^{\infty}\jap{x}^{-\infty},g_h)$ and
\begin{align*} 
&\frac{\pa}{\pa t} F_k(h,t) + i[P,F_k(h,t)]  \geq   -\Op( r_k(h,t,\cdot,\cdot)),\\
&F_{k+1}(h,t) = F_{k}(h,t) + \Op(\f_{k,h,t}),\\
&r_{k,h,t}(x,\x) \leq C_{k+1}h^{\frac{k-m+2}{m-1}} \g_{k+1,h,t}(x,\x)\quad\text{modulo } S(h^{\infty}\jap{x}^{-\infty},g_h).
\end{align*}
By the Borel's Theorem (see \cite{Z} Theorem 4.15), we can define
\[
\f_{h,t}(x,\x) \sim \sum_{n=0}^{\infty} \f_{j,h,t}(x,\x)
\]
and
$F(h,t) = \Op(\f_{h,t})$. 
Then, $F(h,t)$ satisfies the properties in Theorem \ref{sing}. This completes the proof of 
Theorem \ref{sing}.
\end{proof}

\section{Compactly supported perturbation}\label{appcpt}

The proof is considerably simpler if the perturbation is compactly supported, since we do not need 
the argument of Subsection~\ref{global}. Here we discuss the simpler argument for this case. 
We assume that there exists $R>0$ such that $\supp q\subset B_R(0)\times \re^n$, 
where $B_R(0)=\{x\in \re^n\mid |x|<R\}$. We note still the local regularity argument 
(Subsection~\ref{local} and Appendices~\ref{appa}, \ref{appb}). 
Let $\g\in C^{\infty}(\re^n)$ be a real-valued function such that $\g=1$ on $\re^n\setminus B_{R+1}(0)$ and $\g=0$ on $B_R(0)$.

\begin{prop}\label{cptsupp}
Let $k\geq 0$ and $u\in L^2(\re^n) \cap H^{k+m-1}_\mathrm{loc}(\re^n)$ be a distributional solution to $(P - i) u = 0$. Then we have $\g u \in H^k$. In particular, $u\in H^k$ follows.
\end{prop}

\begin{proof}
Set $N = I - \Delta$ and $N_{\e} = (I - \Delta)(I - \e \Delta)^{-1}$ and define $L =p_0(D)$ where $\Delta$ denotes the standard Laplacian on $\re^n$. By virtue of the support property of $\g$, we compute
\begin{align*}
L(\g u) =  P(\g u) = \g Pu + [P, \g ]u
=  i\g u + Ku,
\end{align*}
where $K\coloneqq [P, \g]$ is compactly supported coefficients differential operator with order $m-1$. We note $Ku \in H^1$ since $u\in H^m_\mathrm{loc}(\re^n)$. Hence, we have
\begin{align*}
2i\Im(N_\e^{2k}(\g u), L(\g u))_{L^2}
= & 2i\Im(N_\e^{2k}(\g u), i\g u + Ku)_{L^2}\\
= & 2i\| N_\e^{k} (\g u) \|_{L^2}^2 + 2i\Im (N_\e^{2k} (\g u), Ku)_{L^2}.
\end{align*}
On the other hand, by the Plancherel theorem, we have
\[
2i\Im (N_\e^{2k} (\g u), L(\g u))_{L^2}
= (N_\e^{2k} (\g u), L(\g u))_{L^2} - (L(\g u), N_\e^{2k} (\g u))_{L^2}=0.
\]
Thus, we have
\[
\| N_\e^{k} (\g u) \|_{L^2}^2
\leq  | \Im (N_{\e}^{2k} ( \g u), Ku) | \leq  \| N_\e^{k} (\g u)\|_{L^2} \| N_\e^{k} Ku \|_{L^2} 
\]
Consequently, take $\e \to 0$  and we obtain
$\| N^{k} (\g u)\|_{L^2}\leq \| N^{k} Ku \|_{L^2}< \infty$, 
by using the monotone convergence theorem and $Ku \in H^k$. This implies $\g u \in H^k$.
\end{proof}

\begin{proof}[Proof of Proposition \ref{reg}]
Suppose that $u\in L^2(\re^n)$ satisfies $(P-i)u=0$. By Proposition \ref{locre}, we have $u\in C^{\infty}(\re^n)\subset H^{3(m-1)/2}_\mathrm{loc}(\re^n)$. By Proposition~\ref{cptsupp}, we conclude $u\in H^{(m-1)/2}\subset H^{(m-1)/2, -1/2}$.
\end{proof}


\end{document}